\documentclass[11pt]{article}

\usepackage{amsmath,amsthm,amscd,latexsym}
\usepackage{amsfonts,pdfsync,color,graphicx}
\usepackage[psamsfonts]{amssymb}
\usepackage{epsfig}
\usepackage{color}
\usepackage{hyperref}
\usepackage[noadjust]{cite}
\usepackage{accents}
\usepackage{float}
\usepackage{listings}

\usepackage[top=1.5in,bottom=1.5in,left=1.3in,right=1.0in]{geometry}

\newcommand{\csch}{\operatorname{csch}}

\newcommand{\R}{{\mathbb R}}

\def\qed{\hbox to 0pt{}\hfill$\rlap{$\sqcap$}\sqcup$}

\newtheoremstyle{mystyle}               
  {}                
  {}                
  {}        
  {}                
  {\bfseries \itshape}       
  {.}      
  { }      
  {}       

\newtheorem{theorem}{Theorem}[section]

\newtheorem{lemma}[theorem]{Lemma} 
\newtheorem{corollary}[theorem]{Corollary}
\theoremstyle{definition}

\theoremstyle{mystyle}

\title{Constant sign Green's function of a second order perturbed periodic problem\thanks{Partially supported by Xunta de Galicia (Spain), project EM2014/032 and AIE, Spain and FEDER, grant PID2020-113275GB-I00}}
\author{Alberto Cabada, Luc{\' i}a L\'opez-Somoza and Mouhcine Yousfi\\
	CITMAga, 15782 Santiago de Compostela, Galicia, Spain,\\ Departamento de Estat\'{\i}stica, An\'{a}lise Matem\'{a}tica e Optimizaci\'on, \\Facultade de Matem\'aticas, 
	Universidade de Santiago de Com\-pos\-te\-la, 15782\\ Santiago de Compostela, Galicia, Spain.\\
	alberto.cabada@usc.es; lucia.lopez.somoza@usc.es; yousfi.mouhcine@usc.es}
	
\date{}
\begin{document}
\maketitle
\begin{abstract}
In this paper we are interested in obtaining the exact expression and the study of the constant sign of the Green's function related to a  second order perturbed periodic problem coupled  with integral boundary conditions at the extremes of the interval of definition. 

To obtain the expression of the Green's function related to this problem we use the theory presented in \cite{CLY} for general non-local perturbed boundary value problems. Moreover, we will characterize the parameter's set where such Green's function has constant sign. To this end, we need to consider first a related second order problem without integral boundary conditions, obtaining the properties of its Green's function and then using them to compute the sign of the one related to the main problem.
\end{abstract}

\section{Introduction}
In this paper we will study the regions of constant sign of the Green's function related to the following perturbed second order periodic problem, coupled with integral conditions on the boundary
\begin{equation}\label{eg}
	\left\{
	\begin{aligned}
		u''(t)+M u(t)&=\sigma(t),\;\; t\in I:=I,\\
		u(0)-u(1)&=\delta_1 \displaystyle \int_{0}^{1} u(s)ds,\\
		u'(0)-u'(1)&=\delta_2 \displaystyle \int_{0}^{1} u(s)ds,
	\end{aligned}
	\right.
\end{equation}
where $M,\delta_1, \delta_2\in \mathbb{R}$. In particular, we will consider separately the cases $M=0$, $M>0 $ and $M<0$ and we will analyze each of them and give the optimal values on $M,\delta_1, \delta_2\in \mathbb{R}$ for which the Green's function (denoted by $G_{M,\delta_1,\delta_2}$) has constant sign.

The interest of this study relies on the fact that the constant sign of the Green's function is fundamental to ensure the existence of constant sign solutions of related nonlinear problems as it is a basic assumption to apply some classical methods as lower and upper solutions, monotone iterative techniques, Leray-Schauder degree theory or fixed point theorems on cones.

Furthermore, the solvability of differential equation coupled with different types of boundary value conditions is a topic that has awaken interest in the recent literature. In particular integral boundary conditions have been widely considered in many works in the recent literature. For this topic, we refer the reader to \cite{CI,Cabada-Jebari, Hu-Yan,Khanfer,Mansouri, Zhang-Abella-Feng} (for integral boundary conditions in second and fourth order ODEs) or \cite{Ahmad-Hamdan-Alsaedi-Ntouyas,Ahmadkhanlu, Cabada_Hamdi,  Chandran-Gopalan-Tasneem-Abdeljawad, Duraisamy-Nandha-Subramanian,Rezapourfrac} (for fractional equations) and the references therein.

In a recent paper (\cite{CLY}) the authors have proved the existence of a relation between the Green's function of a differential problem coupled with some functional boundary condition (where the functional is given by a linear operator) and the Green's function of the same differential problem coupled with homogeneous boundary conditions. Such formula will be used now to compute the expression of the Green's function related to problem \eqref{eg} for the cases $M>0$ and $M<0$. In such cases, the very well-known properties of the periodic Green's function will help to study the constant sign of the Green's function of problem \eqref{eg}. For the case $M=0$ this technique cannot be applied, as $M=0$ is an eigenvalue of the periodic problem and, consequently, we will need to compute the expression of the Green's function of \eqref{eg} by means of direct integration in this case.

The paper is organized as follows: in Section 2, we compile the preliminary results that will be used later.  In next section we prove some properties of the Green's function which allow us to simplify the study of the general case. Section 4 considers the particular case of considering parameter $\delta_1=0$ in problem \eqref{eg}. Finally, Section 5 includes the complete study of the case $\delta_1\neq 0$, which is related to the study developed in Section 4 by means of the general properties proved in Section 3.

\section{Preliminaries }
In this section we compile the main results of \cite{CLY} that are then used to develop the rest of the paper.

Consider the following  $n$-th order linear boundary value problem with parameter dependence:
\begin{equation}\label{1}
\left\{\begin{split}
T_n\left[M\right]u(t)&=\sigma(t),\quad t\in J:=\left[a,b\right],\\
B_{i}(u)&=\delta_{i}\, C_i(u), \quad i=1,\ldots,n,
\end{split}
\right.
\end{equation}
where $T_n\left[M\right]u(t):=L_{n} u(t)+M\,u(t)$, $t \in J$,
with 
\begin{equation*}
L_{n} u(t):=u^{\left(n\right)}(t)+a_{1}(t) u^{\left(n-1\right)}(t)+\cdots +a_{n}(t) u(t),\quad t\in J.
\end{equation*}
Here   $\sigma$ and $a_{k}$ are continuous functions for all $k=0,\ldots ,n-1$, $M\in \mathbb{R}$ and $\delta_{i}\in \mathbb{R}$ for all $i=1,\ldots,n$. Moreover, $C_i:C(I)\rightarrow \mathbb{R}$ is a linear continuous operator and $B_i$ covers the general two point linear boundary conditions, i.e.:
\begin{equation*}
B_{i}\left(u\right)=\displaystyle \sum_{j=0}^{n-1} \left(\alpha_{j}^{i} u^{\left(j\right)}\left(a\right)+\beta_{j}^{i} u^{\left(j\right)}\left(b\right) \right),\quad i=1,\ldots ,n,
\end{equation*}
being $\alpha_{j}^{i},\, \beta_{j}^{i}$  real constants for all  $i=1,\ldots ,n,\, j=0,\ldots ,n-1$. 

We note that problem \eqref{eg} is a particular case of \eqref{1}.

\begin{lemma}\label{l-ex-green} \cite[Lemma 1]{CLY}
	There exists the unique Green's function related to the homogeneous problem
	\begin{equation}\label{2}
		\left\{
		\begin{aligned}
			T_n\left[M\right] u(t)&=0,\quad t\in J,\\
			B_{i}\left(u\right)&=0,\quad i=1,\ldots n,
		\end{aligned}
		\right.
	\end{equation}
if and only if for any $i\in \{1, \cdots,n\}$, the following problem 
	\begin{equation}\label{e-wj}
		\left\{
		\begin{aligned}
			T_n\left[M\right] u(t)&=0,\quad t\in J,\\
			B_{j}\left(u\right)&=0,\quad j\neq i,\\
			B_{i}\left(u\right)&=1,
		\end{aligned}
		\right.
	\end{equation}
	has a unique solution, that we denote as $\omega_{i}(t)$, $t\in J$.
\end{lemma}

The following result shows the existence and uniqueness of the solution of problem \eqref{1} and it is a direct consequence of \cite[Theorem 2]{CLY}.

\begin{theorem}\cite[Corollary 2]{CLY}
	Assume that the homogeneous problem \eqref{2} has $u=0$ as its unique solution and let $G_{M,0,\dots,0}$ be its unique Green's function. 
	Let $\sigma \in C(J)$, and $\delta_{i},$  $i=1,\dots, n$, be such that $\sum_{i=1}^{n} \delta_{i} \, C(\omega_{i})\neq 1$. Then problem \eqref{1}
	has a unique solution $u\in C^n(J)$, given by the expression
	\begin{equation*}
	u(t)=\displaystyle \int_{a}^{b} G_{M,\delta_1,\dots,\delta_n}(t,s) \, \sigma(s) ds,
	\end{equation*}
	where 
	\begin{equation}\label{14}
	G_{M,\delta_1,\dots,\delta_n}(t,s):=G_{M,0,\dots,0}(t,s)+\dfrac{\displaystyle \sum_{i=1}^{n} \delta_{i} \, \omega_{i}(t)}{1-\displaystyle \sum_{j=1}^{n} \delta_{j} \, C(\omega_{j})} \, C(G_{M,0,\dots,0}(\cdot,s)).
	\end{equation}
\end{theorem}

\section{First results}

This section is devoted to deduce some preliminary results that will be fundamental in the development of the paper. In a first moment, we deduce the following symmetric property.
\begin{lemma}
	\label{G-symmetric}
	Assume that problem~\eqref{eg} has a unique solution and let $G_{M,\delta_1,\delta_2}$ be its related Green's function. Then the following symmetric property holds:	
	\begin{equation}\label{dadah}
		G_{M,\delta_1,\delta_2}(t,s)=G_{M,-\delta_1,\delta_2}(1-t,1-s).
	\end{equation} 
\end{lemma}
\begin{proof}
	Let 
	\[u(t)=\int_0^1{G_{M,\delta_1,\delta_2}(t,s)\, \sigma(s)\, ds}\]
	be the unique solution of problem~\eqref{eg}.
	
	It is immediate to verify that $v(t):=u(1-t)$ is the unique solution of the following problem:
	\begin{equation*}
		\left\{
		\begin{split}
			v''(t)+M v(t)&=\sigma\left(1-t\right), \quad t\in I,\\
			v(0)-v(1)&=-\delta_1 \displaystyle \int_{0}^{1} v(s)ds,\\
			v'(0)-v'(1)&=\delta_2 \displaystyle \int_{0}^{1} v(s)ds.
		\end{split}
		\right.
	\end{equation*}
As a direct consequence, we deduce that
	\[v(t)=\int_0^1{G_{M,-\delta_1,\delta_2}(t,s)\, \sigma(1-s)\, ds}.\]
	
	On the other hand, we have
	\begin{equation*}
		v(t)=u(1-t)=\int_0^1{G_{M,\delta_1,\delta_2}(1-t,s)\, \sigma(s)\, ds}= \int_0^1{G_{M,\delta_1,\delta_2}(1-t,1-s)\, \sigma(1-s)\, ds}.
	\end{equation*}
	Therefore, the equality \eqref{dadah} is fulfilled directly by identifying the two previous equalities.
\end{proof}

Let us now characterize the points where a constant sign Green's function may vanish.
\begin{lemma}\label{L:cs_van_gen}
Let $M<\pi^2$. If $G_{M,\delta_1,\delta_2}$ has constant sign on $I\times I$ and vanishes at some point $(t_0,s_0)$, then either $t_0=0$, $t_0=1$ or $t_0=s_0$.
\end{lemma}
\begin{proof}
Let us suppose that $(t_0,s_0)\in (0,1)\times [0,1)$, with $t_0> s_0$. In such a case, $u(t)=G_{M,\delta_1,\delta_2}(t,s_0)$ solves the problem
	\begin{equation*}\left\{\begin{split}
	&u''(t)+M\,u(t)=0, \quad a.\,e. \ t\in (s_0,1], \\
	&u(t_0)=u'(t_0)=0,
	\end{split}\right.\end{equation*}
	and so $G_{M,\delta_1,\delta_2}(t,s_0)=0$ for all $t\in (s_0,1]$. This is a contradiction with Sturm's comparison results, \cite{MW}, as for $M<\pi^2$ the distance between two consecutive zeros of any solution of the equation $u''(t)+M\,u(t)=0$ must be bigger than $1$.
	
	We note that the case $(t_0,s_0)\in (0,1)\times (0,1]$, with $t_0< s_0$ can also be discarded as Lemma~\ref{G-symmetric} implies that if $G_{M,\delta_1,\delta_2}$ has constant sign and vanishes at $(t_0,s_0)\in (0,1)\times (0,1]$, with $t_0< s_0$, then  $G_{M,\delta_1,\delta_2}$ will also have constant sign and vanish at the point $(1-t_0,1-s_0)$ (which satisfies that $1-t_0>1-s_0$).
\end{proof}

For $M\in \mathbb{R}\setminus\{0\}$, according to \eqref{14}, the Green's function of problem \eqref{eg} is
\begin{equation}\label{70}\begin{split}
G_{M,\delta_1,\delta_2}(t,s)&=G_{M,0,0}(t,s)+\dfrac{\delta_1 \, \omega_{1}(t)+ \delta_2\, \omega_2(t)}{1-\left( \delta_1 \int_{0}^{1} \omega_1(s)\, ds + \delta_2 \int_{0}^{1} \omega_2(s)\, ds \right)} \, \int_{0}^{1} G_{M,0,0}(t,s) \, dt ,
\end{split}\end{equation}
where  $\omega_{1}$ is the unique solution to the problem 
\begin{equation*}
	\left\{
	\begin{aligned}
		u''(t)+M u(t)&=0,\;\; t\in I,\\
		u(0)-u(1)&= 1,\\
		u'(0)-u'(1)&=0,
	\end{aligned}
	\right.
\end{equation*}
and  $\omega_{2}$ is the unique solution to the problem 
\begin{equation*}
	\left\{
	\begin{aligned}
		u''(t)+M u(t)&=0,\;\; t\in I,\\
		u(0)-u(1)&= 0,\\
		u'(0)-u'(1)&=1.
	\end{aligned}
	\right.
\end{equation*}
It is immediate  to see that and $\omega_{1}(t)=\omega_{2}'(t)$, for all $t\in I$. As a consequence $\int_{0}^{1} \omega_1(s)\, ds =0$. Moreover, it is very well known (see \cite{C1,C2}) that $\omega_{2}(t)=G_{M,0,0}(t,0)$ and
\begin{equation*}
G_{M,0,0}(t,s)=	\left\{
	\begin{array}{ll}
		G_{M,0,0}(t-s,0), & 0\le s \le t \le 1,\\ \\
G_{M,0,0}(1+t-s,0),& 0\le t<s \le 1.
	\end{array}
	\right.
\end{equation*}

Thus, it holds that 
\[\int_{0}^{1} G_{M,0,0}(t,s) dt= \int_{0}^{1} G_{M,0,0}(t,0) dt= \int_{0}^{1} \omega_2(t) dt = -\frac{1}{M} \int_{0}^{1} \omega_2''(t) dt = \frac{1}{M}, \quad \forall\, s\in I. \]
As a consequence, \eqref{70} can be rewritten as
\begin{equation}\label{70_bis}\begin{split}
		G_{M,\delta_1,\delta_2}(t,s)&=G_{M,0,0}(t,s)+\dfrac{\delta_1 \, \omega_{1}(t)+ \delta_2\, \omega_2(t)}{M-\delta_2} = G_{M,0,\delta_2}(t,s) +\dfrac{\delta_1 \, \omega_{1}(t)}{M-\delta_2}.
\end{split}\end{equation}

Taking into account previous expression, we will start with the study of the case $\delta_1=0$.

\section{Study of case $\delta_1=0$} 
In this section we will study the regions of constant sign of the Green's function related to the following perturbed periodic problem
\begin{equation}\label{fous}
\left\{
\begin{aligned}
u''(t)+M u(t)&=\sigma(t),\;\; t\in I,\\
u(0)-u(1)&=0,\\
u'(0)-u'(1)&=\delta_2 \displaystyle \int_{0}^{1} u(s)ds,
\end{aligned}
\right.
\end{equation}
for $M, \delta_2 \in \mathbb{R}$.

It is immediate to verify that  the spectrum of problem \eqref{fous} is given by
\[(\delta_2,M) \in \left\{\ (4k^2 \pi^2, \delta_2), \ \delta_2\in \mathbb{R}, \ k=1,2,\dots\right\} \cup \left\{(M,M), \ M\in \mathbb{R} \right\}.\]
 
On the other hand, the spectrum of the homogeneous periodic problem ($\delta_1=\delta_2=0$)
 \begin{equation}\label{nons}
	\left\{
	\begin{aligned}
		u''(t)+M u(t)&=\sigma(t),\;\; t\in I,\\
		u(0)-u(1)&=0,\\
		u'(0)-u'(1)&=0,
	\end{aligned}
	\right.
\end{equation} 
is given by $4k^2 \pi^2$, $k=0,1,2\ldots$, that is, $G_{M,0,0}$ exists and is unique if and only if $M\neq 4k^2 \pi^2$, $k=0,1,2\ldots$.

Thus, formula \eqref{70} is valid to compute $G_{M,0,\delta_2}$ for all $M\neq 4 k^2 \pi^2,\;\; k=0,1,\ldots$ and $\delta_2 \neq M$. The Green's function $G_{0,0,\delta_2}$, with $\delta_2\neq 0$, exists but it can not be calculated using $\eqref{70}$, so we will do it by direct integration.

Let us now characterize the points where a constant sign Green's function may vanish.
 \begin{lemma}\label{L:nonneg_van}
 	Let $M<\pi^2$. If $\delta_2<0$, $G_{M,0,\delta_2}$ is non-negative on $I\times I$ and vanishes at some point $(t_0,s_0) \in I \times I$, then $t_0=s_0$.
 \end{lemma}
 \begin{proof}
From Lemma~\ref{L:cs_van_gen} we only need to discard the cases $(0,s_0)$ and $(1,s_0)$ with $s_0\in (0,1)$. We note that, since $G_{M,0,\delta_2}(0,s_0)=G_{M,0,\delta_2}(1,s_0)$, both cases are equivalent. Suppose then that
 	\[G_{M,0,\delta_2}(0,s_0)=G_{M,0,\delta_2}(1,s_0)=0. \]
 	In such a case, it would occur that $\frac{\partial\, G_{M,0,\delta_2}}{\partial\,t} (0,s_0)\ge 0$ and $\frac{\partial\, G_{M,0,\delta_2}}{\partial\,t} (1,s_0)\le 0$, which contradicts the fact that
 	\[\frac{\partial\, G_{M,0,\delta_2}}{\partial\,t} (0,s) - \frac{\partial\, G_{M,0,\delta_2}}{\partial\,t} (1,s)=\frac{\delta_2}{M-\delta_2}<0 \quad \forall\,s\in (0,1). \] 
 	
 	As a consequence, the only possibility is that $t_0=s_0$.
 \end{proof}

 \begin{lemma}\label{L:nonpos_van}
 	Let $M<\pi^2$. If $\delta_2>M$,  $G_{M,0,\delta_2}$ is non-positive on $I\times I$ and vanishes at some point $(t_0,s_0) \in I \times I$, then either  $t_0=0$ or $t_0=1$.	
 \end{lemma}
 \begin{proof}
From Lemma~\ref{L:cs_van_gen} we only need to discard the case $t_0=s_0$. In such a case, since $G_{M,0,\delta_2}$ is non-positive, it must occur that $\frac{\partial\, G_{M,0,\delta_2}}{\partial\,t} (t_0^-,t_0)\ge 0$ and $\frac{\partial\, G_{M,0,\delta_2}}{\partial\,t} (t_0^+,t_0)\le 0$, which contradicts the fact that
 	\[\frac{\partial\, G_{M,0,\delta_2}}{\partial\,t} (t^+,t) - \frac{\partial\, G_{M,0,\delta_2}}{\partial\,t} (t^-,t)=1  \quad \forall\, t\in(0,1). \] 
 	
 	As a consequence, the only possibility is that either $t_0=0$ or $t_0=1$.
 \end{proof}

\subsection{Expression of the Green's function}

Now, we obtain the exact expression of the Green's function related to problem \eqref{fous} by considering the different situations of the parameters $M$ and $\delta_2$. We start with $M \neq 0$, i.e., the situation in which problem \eqref{fous} is uniquely solvable for $\delta_2=0$.

We point out that the expressions of the Green's function $G_{M,0,0}(t,s)$ are deduced from reference \cite {programa_green}, where it has been constructed and algorithm that calculates the exact expression of the Green's function related to any $n$th order differential equation, with constant coefficients coupled to arbitrary homogeneous ($\delta_i=0, \; i=0, \ldots, n-1$) two-point linear boundary conditions. Such algorithm has been developed in a Mathematica package that is available at \cite{programa_green-wolfram}.

\subsubsection{$M\neq 0$}
In such a case, using expression \eqref{70_bis} and taking into account that  $\omega_2(t)=G_{M,0,0}(t,0)$ (see \cite{C1,C2}), the expression of the  Green's function related to problem~\eqref{fous} is given by
\begin{equation}\label{eq:GM02}
	G_{M,0,\delta_2}(t,s)=G_{M,0,0}(t,s)+ \dfrac{\delta_2 }{M-\delta_2}\, G_{M,0,0}(t,0).
\end{equation}

We shall consider two different cases:\\

\underline{ $M=m^{2}>0$, with $m\in (0,\infty)$:}\\

In such a case $G_{M,0,0}$ is given by the expression  
\begin{equation*}
G_{M,0,0}(t,s)=\dfrac{\csc\left(\frac{m}{2}\right)}{2m}\left\{
\begin{aligned}
& \cos \left(\frac{m}{2} \left(1+2s-2t\right)\right) ,\;\; 0\leq s\leq t\leq 1,\\
& \cos \left(\frac{m}{2} \left(1+2t-2s\right)\right) ,\;\; 0\le t<s\leq 1,
\end{aligned}
\right.
\end{equation*}
and so \eqref{eq:GM02} implies that

\begin{equation*}
	G_{M,0,\delta_2}(t,s)=\dfrac{\delta_2 \cos \left(\frac{m}{2} \left(1-2t\right)\right) }{m^{2} -\delta_2}+\frac{\csc \left(\frac{m}{2}\right)}{2m}\left\{
	\begin{aligned}
		& \cos \left(\frac{m}{2} \left(1+2s-2t\right)\right) ,\;\; 0\leq s\leq t \leq 1,\\
		& \cos \left(\frac{m}{2} \left(1+2t-2s\right)\right) ,\;\; 0\le t<s\leq 1.
	\end{aligned}
	\right.
\end{equation*}

\underline{  $M=-m^{2}<0$, with $m\in (0,\infty)$:}\\

In this case $G_{M,0,0}$  is given by
\begin{equation*}
G_{M,0,0}(t,s)=\frac{1}{2m\left(1-e^{m}\right)}
\begin{cases}
e^{m\left(1+s-t\right)}+e^{m\left(t-s\right)}, & 0\leq s\leq t\leq 1,\\
e^{m\left(1+t-s\right)}+e^{m\left(s-t\right)}, & 0\le t<s\leq 1,
\end{cases}
\end{equation*}
and thus

\begin{equation*}
	G_{M,0,\delta_2}(t,s)=-\frac{\delta_2}{m^2+\delta_2} \left(e^{m\left(1-t\right)}+e^{mt}\right)+\frac{1}{2m\left(1-e^{m}\right)} \left\{
	\begin{aligned}
		&e^{m\left(1+s-t\right)}+e^{m\left(t-s\right)},\;\; 0\leq s\leq t\leq 1,\\
		&e^{m\left(1+t-s\right)}+e^{m\left(s-t\right)},\;\; 0\le t<s\leq 1.
	\end{aligned}
	\right.
\end{equation*}

\subsubsection{$M=0$}
In this case, formula \eqref{14} is not valid to calculate the expression of the Green's function, so we shall compute it by direct integration. The solution of equation $u''(t)=\sigma(t)$ is given by 
\begin{equation*}
	u(t)=c_{1}+c_{2} t+\displaystyle \int_{0}^{t} \left(t-s\right) \sigma(s)ds.
\end{equation*}
Then, $u'(t)=c_{2}+ \int_{0}^{t} \sigma(s) ds$. Imposing condition $u(0)=u(1)$, we have that $c_{2}=-\int_{0}^{1} \left(1-s\right) \sigma(s)ds $. Therefore, $u'(0)-u'(1)=- \int_{0}^{1} \sigma(s) ds$. Since $u'(0)-u'(1)=\delta_2 \int_{0}^{1} u(s) ds$ we deduce that 
\begin{equation*}
	c_{1}= -\frac{1}{\delta_2} \int_{0}^{1} \sigma(s) ds -\frac{1}{2} \displaystyle \int_{0}^{1} \left(1-s^{2}\right) \sigma(s) ds+\displaystyle \int_{0}^{1} \left(s-s^{2}\right) \sigma(s) ds+ \frac{1}{2} \displaystyle \int_{0}^{1} \left(1-s\right) \sigma(s) ds. 
\end{equation*}
So, 
\begin{equation*}
	\begin{aligned}
		u(t)=& -\frac{1}{\delta_2} \int_{0}^{1} \sigma(s) ds -\frac{1}{2} \displaystyle \int_{0}^{1} \left(1-s^{2}\right) \sigma(s) ds + \int_{0}^{1} s\left(1-s\right) \sigma(s) ds   \\
		&+ \left(\frac{1}{2}-t\right)  \int_{0}^{1} \left(1-s\right) \sigma(s) ds + \int_{0}^{t} \left(t-s\right) \sigma(s) ds,\\
		=&  \int_{0}^{1} G_{0,0,\delta_2}(t,s) \, \sigma(s) ds,   
	\end{aligned}
\end{equation*}
where 
\begin{equation*}
	G_{0,0,\delta_2}(t,s)=\left\{
	\begin{aligned}
		& -\frac{s}{2}-\frac{s^{2}}{2} +st-\frac{1}{\delta_2},\;\; 0\leq s\leq t\leq 1,\\
		& \frac{s}{2}-\frac{s^{2}}{2}-t+st-\frac{1}{\delta_2},\;\; 0\le t<s\leq 1.
	\end{aligned}
	\right.
\end{equation*}

\subsection{Regions of constant sign of the Green's function}
We shall study now the regions in which previous functions have constant sign. First we note that we can bound these regions in the following way.
\begin{lemma}
	$G_{M,0,\delta_2}$ will never have constant sign  on $I \times I$ for all $M>\pi^2$.	
\end{lemma}
\begin{proof}
	 From expression \eqref{eq:GM02} and the fact that
	\begin{equation*}
		G_{M,0,\delta_2}(t,0)=\left(1+\frac{\delta_2}{M-\delta_2}\right) G_{M,0,0}(t,0) =\frac{m^2}{m^{2}-\delta_2}\, \dfrac{\csc\left(\frac{m}{2}\right)\cos\left(\frac{m}{2} \left(1-2t\right)\right)}{2m},
	\end{equation*}
	it is immediately deduced that $G_{M,0,\delta_2}(t,0)$ is sign-changing on $I$ for any  $m>\pi$.
\end{proof}

\begin{lemma}
	The following properties are fulfilled:
	\begin{itemize}
		\item If $M<\delta_2\le 0$ then $G_{M,0,\delta_2}$ is negative on $I \times I$.
		\item If  $0\le\delta_2<M \le \pi^2$ then $G_{M,0,\delta_2}$ is positive on $I \times I$. 
		\item If $M=\pi^2$ and $0\le \delta_2<M$ then $G_{M,0,\delta_2}$  vanishes at the set $A:=\{(0,0),(0,1),(1,0),(1,1)\}$ and is positive on $(I \times I)\backslash A$.
	\end{itemize}
\end{lemma}
\begin{proof}
	It is immediately deduced from \eqref{eq:GM02} and the fact that $G_{M,0,0}$ is negative on $I \times I$ for $M<0$, positive on $I \times I$ for $0<M<\pi$,  and positive on $(I \times I)\backslash A$, vanishing at the set $A$, for $M=\pi^2$.
\end{proof}

Moreover,
\[\dfrac{\partial G_{M,0,\delta_2}}{\partial \delta_2}(t,s)= \frac{M}{(M-\delta_2)^2}\, G_{M,0,0}(t,0)>0, \quad \forall\, M\in\left(-\infty,\pi^2\right)\setminus\{0\}, \ \delta_2\neq M, \ t,s\in I\]
and
\begin{equation}\label{e:derdelta2_cero}
\frac{\partial G_{0,0,\delta_2}}{\partial \delta_2}(t,s)= \left(\frac{1}{\delta_2}\right)^2>0 \quad \forall\,t,s\in I.
\end{equation}
As a consequence, for any fixed $M<\pi^2$,  $G_{M,0,\delta_2}$ is strictly increasing with respect to $\delta_2$ and so we deduce the following facts:
\begin{itemize}
	\item Since $G_{M,0,0}>0$ on $I \times I$ for $M\in(0,\pi^2)$, we know that  $G_{M,0,\delta_2}$ will be positive for some values of $\delta_2<0$. In particular, $G_{M,0,\delta_2}$ will be positive for $\delta_2 \in (\delta_2(M),0]$, where the optimal value $\delta_2(M)$  will be either  $-\infty$ or the biggest negative real value  for which $G_{M,0,\delta_2(M)}$ attains the value  zero at some point  $\left(t_{0},s_{0}\right)\in I\times I$.
	\item Since $G_{M,0,0}<0$ on $I \times I$ for $M<0$, we know that  $G_{M,0,\delta_2}$ will be negative for some values of $\delta_2>0$. In particular, $G_{M,0,\delta_2}$ will be negative for $\delta_2 \in [0,\delta_2(M))$, where the optimal value $\delta_2(M)$  will be either  $+\infty$ or the smallest positive real value  for which $G_{M,0,\delta_2(M)}$ attains the value  zero at some point  $\left(t_{0},s_{0}\right)\in I\times I$.
\end{itemize} 

Let us study now the range of values $\delta_2<0$ for which $G_{M,0,\delta_2}$ is positive.
\begin{theorem}\label{th:Mdelt2_pos}
	If $M=m^2$ with $m\in \left(0,\pi\right)$ and $\delta_2\le 0$, then $G_{M,0,\delta_2}(t,s)>0$ for all $(t,s)\in I\times I$  if and only if \[-\frac{m^{2}\cos\left(\frac{m}{2}\right)}{1-\cos\left(\frac{m}{2}\right)}< \delta_2  \le 0.\]
\end{theorem}
\begin{proof}
	From Lemma \ref{L:nonneg_van}, we only need to study the values of function $G_{M,0,\delta_2}$ at the diagonal of the square of definition, where we get the function 
	\begin{equation*}
		h(t)=G_{M,0,\delta_2}(t,t)=\dfrac{\coth\left(\frac{m}{2}\right)}{2m}+\dfrac{\delta_2 \cos\left(\frac{m}{2} \left(1-2t\right)\right) \csc\left(\frac{m}{2}\right)}{2m\left(m^{2}-\delta_2\right)}, \;\; t\in I,
	\end{equation*}
	whose minimum is attained at $t=\frac{1}{2}$. Therefore, $h$ has positive sign on $I$ if and only if $h\left(\frac{1}{2}\right)$ is positive, that is, $\delta_2>-\frac{m^{2}\cos\left(\frac{m}{2}\right)}{1-\cos\left(\frac{m}{2}\right)}$.
\end{proof}

Let us analyse now the range of values $\delta_2>0$ for which $G_{M,0,\delta_2}$ is negative.
\begin{theorem}\label{th:Mdelt2_neg}
	Let $M=-m^{2}$ with $m\in(0,\infty)$ and $\delta_2 \ge 0$, then $G_{M,0,\delta_2}$ is strictly negative  on $I\times I$ if and only if \[0\le \delta_2<\dfrac{2m^{2}e^{\frac{m}{2}}}{1+e^{m}-2e^{\frac{m}{2}}}.\]
\end{theorem}
\begin{proof}
	From Lemma \ref{L:nonpos_van}, we only need to study the values of function $G_{M,0,\delta_2}$ at the points of the form $\left(0,s\right)$ and $\left(1,s\right)$. So, we have to study the function
	\begin{equation*}
		r(s)=G_{M,0,\delta_2}(0,s)=G_{M,0,\delta_2}(1,s)=\frac{1}{2m\left(1-e^{m}\right)}\left(e^{m\left(1-s\right)}+e^{ms}-\frac{\delta_2}{\delta_2+m^{2}} \left(1+e^{m}\right) \right),
	\end{equation*}
	whose maximum value is attained at $s=\frac{1}{2}$. Therefore, $r$ is negative if and only if $r\left(\frac{1}{2}\right)<0$, that is, $\delta_2<\frac{2m^{2}e^{\frac{1}{2}}}{1+e^{m}-2e^{\frac{m}{2}}}$.
\end{proof}

From previous results and \eqref{e:derdelta2_cero}, we deduce the following facts:
\begin{itemize}
	\item Since $G_{M,0,\delta_2}>0$ for $M\in(0,\pi^2)$ and $-\frac{m^{2}\cos\left(\frac{m}{2}\right)}{1-\cos\left(\frac{m}{2}\right)}< \delta_2 \le 0$, we know that  $G_{0,0,\delta_2}$ will be positive for some values of $\delta_2<0$. In particular, $G_{0,0,\delta_2}$ will be positive for $\delta_2 \in (\delta_2(0),0)$, where the optimal value $\delta_2(0)$  will be either  $-\infty$ or the biggest negative real value  for which $G_{0,0,\delta_2(0)}$ attains the value  zero at some point  $\left(t_{0},s_{0}\right)\in I\times I$.
	\item Since $G_{M,0,\delta_2}<0$ for $M<0$ and $0\le \delta_2<\frac{2m^{2}e^{\frac{m}{2}}}{1+e^{m}-2e^{\frac{m}{2}}}$, we know that  $G_{0,0,\delta_2}$ will be negative for some values of $\delta_2>0$. In particular, $G_{0,0,\delta_2}$ will be negative for $\delta_2 \in (0,\delta_2(0))$, where the optimal value $\delta_2(0)$  will be either  $+\infty$ or the smallest positive real value  for which $G_{0,0,\delta_2(0)}$ attains the value  zero at some point  $\left(t_{0},s_{0}\right)\in I\times I$.
\end{itemize} 

Let's study the sign of function  $G_{0,0,\delta_2}$ according to the value of $\delta_2\in \mathbb{R}\setminus\{0\}$.

\begin{theorem}
$G_{0,0,\delta_2}$ is strictly negative on $I\times I$ if and only if $\delta_2\in\left(0,8\right)$.
\end{theorem}
\begin{proof}
	For $\delta_2>0$, using Lemma \ref{L:nonneg_van}, the function to study in this case is 
	\begin{equation*}
		r(s)=G_{0,0,\delta_2}(0,s)=G_{0,0,\delta_2}(1,s)=\frac{s}{2}-\frac{s^{2}}{2}-\frac{1}{\delta_2},\;\; s\in I,
	\end{equation*}
	which reaches its maximum at $s=\frac{1}{2}$. As a consequence, $G_{0,0,\delta_2}$ is negative if and only if $r\left(\frac{1}{2}\right)<0$, that is, $0<\delta_2<8$.
\end{proof}

Using the same arguments, by means of Lemma \ref{L:nonpos_van}, we arrive at the following result for the negative sign of $\delta_2$.
\begin{theorem}
$G_{0,0,\delta_2}$ is strictly positive on $I\times I$ if and only if $\delta_2\in\left(-8,0\right)$.
\end{theorem}

Finally, we have that:
\begin{itemize}
	\item Since $G_{0,0,\delta_2}>0$ on $I \times I$ for  $\delta_2\in(-8,0)$, we know that  $G_{M,0,\delta_2}$ will be positive for some values of  $\delta_2<M<0$. In particular, $G_{M,0,\delta_2}$ will be positive for $\delta_2 \in (\delta_2(M),M)$, where the optimal value $\delta_2(M)$  will be either  $-\infty$ or the biggest negative real value  for which $G_{M,0,\delta_2(M)}$ attains the value  zero at some point  $\left(t_{0},s_{0}\right)\in I\times I$.
	\item Since $G_{0,0,\delta_2}<0$ on $I \times I$ for $\delta_2\in (0,8)$, we know that  $G_{M,0,\delta_2}$ will be negative for some values of $\delta_2>M>0$. In particular, $G_{M,0,\delta_2}$ will be negative for $\delta_2 \in (M,\delta_2(M))$, where the optimal value $\delta_2(M)$  will be either  $+\infty$ or the smallest positive real value  for which $G_{M,0,\delta_2(M)}$ attains the value  zero at some point  $\left(t_{0},s_{0}\right)\in I\times I$.
\end{itemize} 

\begin{theorem}\label{th:Mdelt2_neg2}
	If $M=m^2$ with $m\in \left(0,\pi\right)$, then $G_{M,0,\delta_2}(t,s)<0$ for all $(t,s)\in I\times I$  if and only if \[m^{2}<\delta_2<\dfrac{m^{2}}{1-\cos\left(\frac{m}{2}\right)}.\]
\end{theorem}
\begin{proof}
	Reasoning as before, using Lemma \ref{L:nonpos_van},  let's now look at the value of $\delta_2$ positive such that  $\delta_2>m^{2}, m\in \left(0,\pi\right)$ the  function $G_{M,0,\delta_2}$ is negative. 
	
	At the points of the form $\left(0,s\right)$ and $\left(1,s\right)$  the corresponding function to study is 
	\begin{equation*}
		r(s)=G_{M,0,\delta_2}(0,s)=G_{M,0,\delta_2}(1,s)=\dfrac{\csc\left(\frac{m}{2} \right)}{2m}\left[ \cos \left(\frac{m}{2} \left(1-2t\right) \right)+ \frac{\delta_2}{m^{2}-\delta_2} \cos\left(\frac{m}{2}\right) \right].
	\end{equation*}
	In this case, $r$ has an absolute maximum at $s=\frac{1}{2}$.
	Thus, $r(s)<0$ for all $s\in I$ if and only if  $r\left(\frac{1}{2}\right)<0$, that is, 
	$\delta_2<\frac{m^{2}}{1-\cos\left(\frac{m}{2}\right)}$.
\end{proof}

We will now make a study of the positive sign of $G_{M,0,\delta_2}$ for $m\in (0,\infty)$ and $\delta_2<-m^{2}<0$. 
\begin{theorem}\label{th:Mdelt2_pos2}
	Let $M=-m^{2}$ with $m\in (0,\infty)$, then the Green's function related to problem \eqref{fous} is strictly positive   on $I\times I$ if and only if $\delta_2>-\dfrac{m^2\left(1+e^{m}\right)}{1+e^{m}-2e^{\frac{m}{2}}}$.
\end{theorem}
\begin{proof}
From Lemma \ref{L:nonneg_van} we only must to study the behavior of the Green's function at the points of its diagonal:
	\begin{equation*}
		h(t)=G_{M,0,\delta_2}(t,t)=\frac{e^{m}+1}{2m\left(1-e^{m}\right)}- \frac{\delta_2}{\delta_2+m^{2}} \frac{e^{m\left(1-t\right)}+e^{mt}}{2m\left(1-e^{m}\right)}
	\end{equation*}
	has in this case an absolute minimum at $t=\frac{1}{2}$. So  $h$ is positive on $I$ if and only if $h\left(\frac{1}{2}\right)>0$, that is, 
	$\delta_2>-\dfrac{m^2\left(1+e^{m}\right)}{1+e^{m}-2e^{\frac{m}{2}}}$.
\end{proof}

Figure~\ref{Fig:sign_G} shows the regions where the function $G_{M,0,\delta_2}$ maintains a constant sign.
\begin{figure}[H]
	\centering
	\includegraphics[width=10cm]{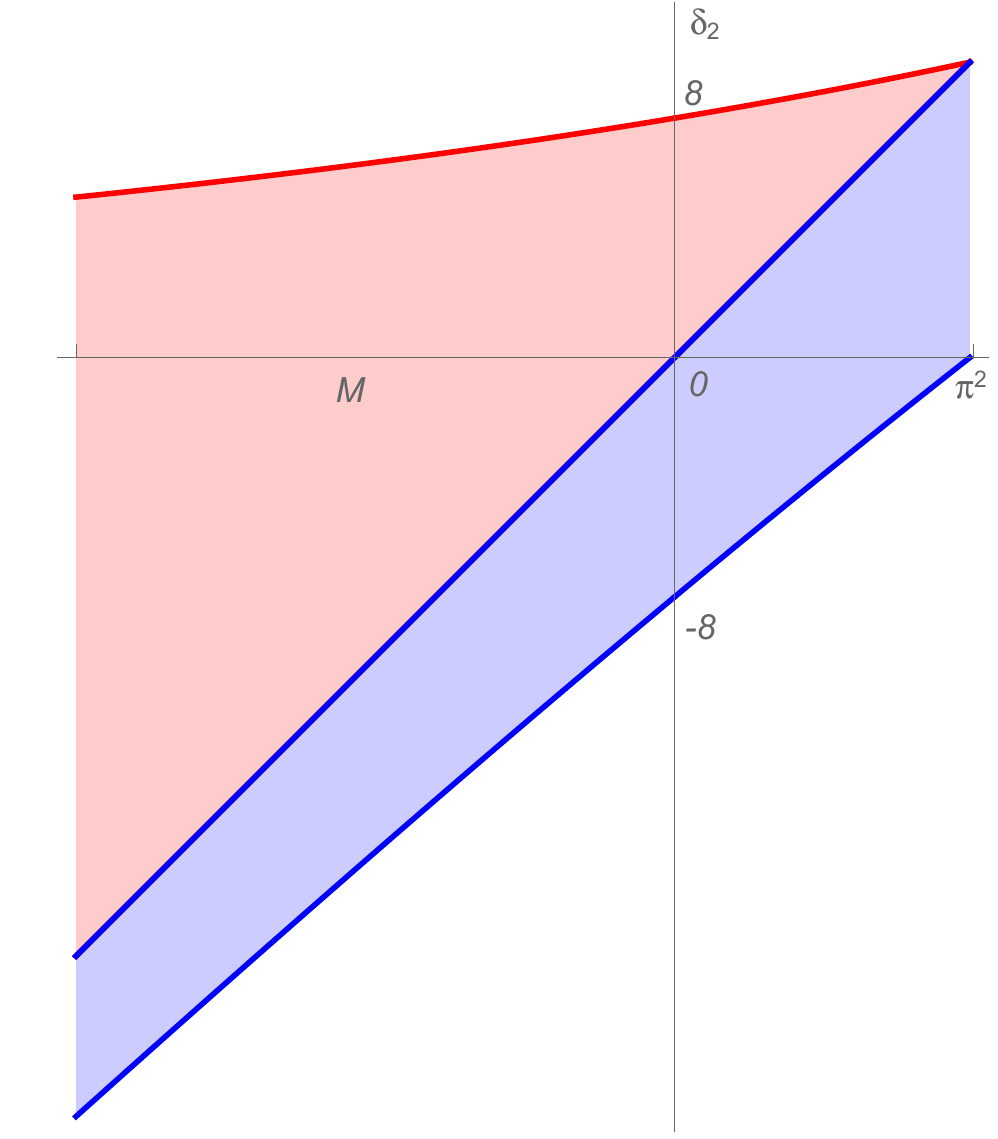}\\
	\caption{The blue region represents the positive sign of $G_{M,0,\delta_2}$ while the red region corresponds to the negative sign of $G_{M,0,\delta_2}$ at the  $(M,\delta_2)$-plane.} \label{Fig:sign_G}
\end{figure}

\section{Case $\delta_1\neq 0$}

In this last section, by continuing the study of the previous case, we consider the general situation of $\delta_1\neq 0$. So, for this case, we  calculate the regions of constant sign of the Green's function related problem~\eqref{eg}. We divide the study in two situations, depending on the fact that the parameter $M$ is or not equals to zero. 

In a first moment we obtain the expression of the Green's function.

\subsection{Expression of the Green's function}
In this subsection we obtain the expression of the Green's function related to the problem \eqref{eg} as a function of the real parameter $M$. 

\subsubsection{$M\neq 0$}
Using formula \eqref{70_bis} and the fact that $\omega_1(t)=\omega_2'(t)$, it is immediate to verify that the expression of  $G_{\delta_1,\delta_2,M}$ is given by
\begin{equation}
	\label{e-G-1-2}
	G_{M,\delta_1,\delta_2}(t,s)=G_{M,0,0}(t,s)+\dfrac{\delta_1 \, \omega_{2}'(t)+ \delta_2\, \omega_2(t)}{M-\delta_2},
	\end{equation}
where
\[\omega_2(t)=\frac{1}{2m}\begin{cases} 
\cos\left(\frac{m}{2}\,(2\,t-1)\right) \csc\left(\frac{m}{2}\right), & M=m^2>0, \ M\neq 4k^2 \pi^2, \ k=0,1,\dots \\ \\
-\cosh\left(\frac{m}{2}\,(2\,t-1)\right) \csch \left(\frac{m}{2}\right), & M=-m^2<0.	
\end{cases}\]
Thus, for $M=m^2>0$, $m>0$, $ M\neq 4k^2 \pi^2, \ k=0,1,\dots $,  $G_{\delta_1,\delta_2,M}$ follows the expression
\begin{equation*}
	G_{M,\delta_1,\delta_2}(t,s)= \frac{\csc\left(\frac{m}{2}\right)}{2m}
	\begin{cases}
\cos\left(m\left(\frac{1}{2}+s-t\right)\right) +\frac{\delta_2 \, \cos\left(m\left(\frac{1}{2}-t\right) \right) +m \delta_1 \, \sin\left(m\left(\frac{1}{2}-t\right) \right)}{m^{2}-\delta_2}, & 0\leq s\leq t\leq 1,\\ \\
\cos\left(m\left(\frac{1}{2}+t-s\right)\right) +\frac{\delta_2\cos\left(m\left(\frac{1}{2}-t\right) \right) +m \delta_1 \,\sin\left(m\left(\frac{1}{2}-t\right) \right) }{m^{2}-\delta_2}, & 0\le t<s\leq 1,
	\end{cases}
\end{equation*}
and for $M=-m^2$, $m>0$, the expression of $G_{\delta_1,\delta_2,M}$ is given by 
\begin{equation*}
\negthickspace	G_{M,\delta_1,\delta_2}(t,s)=\frac{\csch\left(\frac{m}{2}\right)}{2m}  
	\begin{cases}
-\cosh \left(m\left(\frac{1}{2}+s-t\right) \right) + \frac{\delta_2 \cosh \left(m\left(\frac{1}{2}-t\right) \right) - m \, \delta_1 \, \sinh\left(m\left(\frac{1}{2}-t\right)\right)}{m^2+\delta_2},		& 0\le s\le t \le 1\\ \\
-\cosh \left(m\left(\frac{1}{2}+t-s\right) \right) + \frac{\delta_2 \cosh \left(m\left(\frac{1}{2}-t\right) \right) - m \, \delta_1 \, \sinh\left(m\left(\frac{1}{2}-t\right)\right)}{m^2+\delta_2},		& 0\le t<s\leq 1.
	\end{cases}
\end{equation*}

\subsubsection{$M=0$}
For the case $M=0$, we cannot apply formula \eqref{70} and we need to compute $G_{0,\delta_1,\delta_2}$ directly.

It is clear that the solutions of the equation $u''(t)=\sigma(t)$, $t\in I$ are given by the expression
\begin{equation}\label{rous}
	u(t)=c_{1}+c_{2}t+\displaystyle \int_{0}^{t} \left(t-s\right) \sigma(s)ds.
\end{equation}
So, $u(0)-u(1)=-c_{2}+ \int_{0}^{1} \left(s-1\right) \sigma(s) ds$. 

On the other hand,
\begin{equation*}
	\displaystyle \int_{0}^{1} u(t)dt=\displaystyle \int_{0}^{1} \left(c_{1}+c_{2}t+\displaystyle \int_{0}^{t} \left(t-s\right) \sigma(s) ds \right)dt =c_{1}+\frac{c_{2}}{2}+\displaystyle \int_{0}^{1} \int_{0}^{t} \left(t-s\right) \sigma(s) \, ds \, dt.
\end{equation*}
Applying Fubini's Theorem, we have that 
\begin{equation*}
	\displaystyle \int_{0}^{1} \int_{0}^{t} \left(t-s\right) \sigma(s) ds=\displaystyle \int_{0}^{1} \int_{s}^{1} \left(t-s\right) \sigma(s) \, dt \, ds =\displaystyle \int_{0}^{1}\left(\frac{s^{2}+1}{2}-s\right) \sigma(s) ds.
\end{equation*}

Imposing the boundary conditions in \eqref{eg}, we arrive at the following system
\begin{equation*}\begin{split}
		\delta_1 c_{1}+\left(\frac{\delta_1}{2}+1\right) c_{2}&= \int_{0}^{1} \left(s-1\right) \sigma(s) ds-\delta_1  \int_{0}^{1} \left(\frac{s^{2}+1}{2}-s\right) \sigma(s) ds,\\
		\delta_2 c_{1}+\frac{\delta_2}{2} c_{2}&= -\int_{0}^{1} \sigma(s) ds -\delta_2 \int_{0}^{1} \left(\frac{s^{2}+1}{2}-s\right) \sigma(s) ds,
\end{split}\end{equation*}
whose solutions are 
\begin{equation*}\begin{split}
		c_{1}&=-\frac{1}{2} \displaystyle \int_{0}^{1} \left(s-1\right) \sigma(s) ds -\int_{0}^{1} \left(\frac{s^{2}+1}{2}-s\right) \sigma(s) ds
		-\frac{\delta_1+2}{2\,\delta_2}  \int_{0}^{1} \sigma(s) ds, \\
		c_{2}&=\displaystyle \int_{0}^{1} \left(s-1\right) \sigma(s)ds +\frac{\delta_1}{\delta_2} \displaystyle \int_{0}^{1} \sigma(s) ds.
\end{split}\end{equation*}
Substituting $c_{1}$ and $c_{2}$ in \eqref{rous} we have that 
\begin{equation*}
	\begin{aligned}
		u(t)=&\,-\frac{1}{2} \displaystyle \int_{0}^{1} \left(s-1\right) \sigma(s) ds -\int_{0}^{1} \left(\frac{s^{2}+1}{2}-s\right) \sigma(s) ds-\frac{\delta_1+2}{2\,\delta_2} \displaystyle \int_{0}^{1} \sigma(s) ds\\
		&+\displaystyle \int_{0}^{1} t\left(s-1\right) \sigma(s)ds +\frac{\delta_1}{\delta_2} \displaystyle \int_{0}^{1} t \,\sigma(s) ds+\displaystyle \int_{0}^{1} \left(t-s\right) \sigma(s) ds\\
		=&\,\displaystyle \int_{0}^{1} G_{0,\delta_1,\delta_2}(t,s) \, \sigma(s) ds,
	\end{aligned}
\end{equation*}
being 
\begin{equation}
		\label{e-G-0}
	G_{0,\delta_1,\delta_2}(t,s)=\frac{1}{2\,\delta_2} \begin{cases}
		-2+\delta_1\,(-1+2t)-s\, \delta_2\,(1+s-2t), & 	0\leq s\leq t \leq 1, \\ \\
		-2+\delta_1 \,(-1+2t)-\delta_2\,(s-1)\,(s-2t), & 0\le t<s\leq 1.
	\end{cases}
\end{equation}

\subsection{Regions of constant sign of the Green's function}

Now, we are in a position to obtain the regions of constant sign of the Green's function as a function of the parameters $M$, $\delta_1$ and $\delta_2$. 

 To this end, we notice that, by direct differentiation on \eqref{e-G-1-2} and \eqref{e-G-0}, the following identities hold:
\[\frac{\partial }{\partial \delta_1}G_{M,\delta_1,\delta_2}(t,s)= \frac{\omega_1(t)}{M-\delta_2} \quad \text{for } M\neq 0, \ M\neq \delta_2  \]
and
\[\frac{\partial }{\partial \delta_1}G_{0,\delta_1,\delta_2}(t,s)= \frac{1}{\delta_2}\left(t-\frac{1}{2}\right),\]
which implies that $\frac{\partial }{\partial \delta_1}G_{M,\delta_1,\delta_2}$ will change sign depending on $t$. As a consequence, there will be some values of $t$ for which $G_{M,\delta_1,\delta_2}$ will increase with respect to $\delta_1$ and some other values of $t$ for which $G_{M,\delta_1,\delta_2}$ will decrease with respect to $\delta_1$. As an immediate consequence we deduce the following result.
\begin{corollary}\label{cor-sign-delta1}
	The two following properties hold:	
	\begin{itemize}
\item If $M$ and $\delta_2$ are such that $G_{M,0,\delta_2}>0$ on $I \times I$, then $G_{M,\delta_1,\delta_2}$ is either positive or changes its sign on $I \times I$.
\item If $M$ and $\delta_2$ are such that $G_{M,0,\delta_2}<0$ on $I \times I$, then $G_{M,\delta_1,\delta_2}$ is either negative or changes its sign on $I \times I$.
\end{itemize}
\end{corollary}

Furthermore, the following result can be easily verified.
\begin{lemma}
If $M$ and $\delta_2$ are such that $G_{M,0,\delta_2}$ changes sign on $I \times I$, then $G_{M,\delta_1,\delta_2}$ also changes its sign on $I \times I$ for every $\delta_1\in \R$.
\end{lemma}
\begin{proof}
It is immediately verified using arguments similar to Theorems \ref{th:Mdelt2_pos}, \ref{th:Mdelt2_neg}, \ref{th:Mdelt2_neg2} and \ref{th:Mdelt2_pos2}. In particular, it is obtained that:
\begin{enumerate}
\item If $M=m^2\in (0,\pi^2)$ and $\delta_2<-\frac{m^{2}\cos\left(\frac{m}{2}\right)}{1-\cos\left(\frac{m}{2}\right)}$, then $G_{M,\delta_1,\delta_2}\left(\frac{1}{2},\frac{1}{2}\right)<0$ and $G_{M,\delta_1,\delta_2}(0,0)>0$.
\item If $M=m^2\in (0,\pi^2)$ and $\delta_2>\dfrac{m^{2}}{1-\cos\left(\frac{m}{2}\right)}$, then $G_{M,\delta_1,\delta_2}\left(\frac{1}{2},\frac{1}{2}\right)>0$ and ${G_{M,\delta_1,\delta_2}(0,0)<0}$.
\item If $M=-m^2$, with $m\in (0,\infty)$, and $\delta_2>\dfrac{2m^{2}e^{\frac{m}{2}}}{1+e^{m}-2e^{\frac{m}{2}}}$, then $G_{M,\delta_1,\delta_2}\left(\frac{1}{2},\frac{1}{2}\right)>0$ and $G_{M,\delta_1,\delta_2}(0,0)<0$.
\item If $M=-m^2$, with $m\in (0,\infty)$, and $\delta_2<-\dfrac{m^2\left(1+e^{m}\right)}{1+e^{m}-2e^{\frac{m}{2}}}$ then $G_{M,\delta_1,\delta_2}\left(\frac{1}{2},\frac{1}{2}\right)<0$ and $G_{M,\delta_1,\delta_2}(0,0)>0$.
\item If $M>\pi^2$ then $G_{M,\delta_1,\delta_2}(t,0)$ is sign-changing on $I$.
\end{enumerate}
\end{proof}

Moreover, since for any fixed $t\in I$, $G_{M,\delta_1,\delta_2}(t, s)$ is either increasing or decreasing with respect to $\delta_1$, we deduce the following facts:
\begin{itemize}
	\item If $M$ and $\delta_2$ are such that $G_{M,0,\delta_2}>0$ on $I \times I$, then $G_{M,\delta_1,\delta_2}$ will be positive  on $I \times I$ for some values (both positive and negative) of $\delta_1$. In particular, by Lemma~\ref{G-symmetric}, we know that $G_{M,\delta_1,\delta_2}$ will be positive on $I \times I$ for $\delta_1 \in (-\delta_1(\delta_2,M),\delta_1(\delta_2,M))$, where the optimal value $\delta_1(\delta_2,M)$  will be either  $+\infty$ or the smallest positive real value  for which $G_{M,\delta_1(\delta_2,M),\delta_2}$ attains the value  zero at some point.
	\item If $M$ and $\delta_2$ are such that $G_{M,0,\delta_2}<0$ on $I \times I$, then $G_{M,\delta_1,\delta_2}$ will be negative  on $I \times I$ for some values (both positive and negative) of $\delta_1$. In particular, by Lemma~\ref{G-symmetric}, we know that $G_{M,\delta_1,\delta_2}$ will be negative on $I \times I$ for $\delta_1 \in (-\delta_1(\delta_2,M),\delta_1(\delta_2,M))$, where the optimal value $\delta_1(\delta_2,M)$  will be either  $+\infty$ or the smallest positive real value  for which $G_{M,\delta_1(\delta_2,M),\delta_2}$ attains the value  zero at some point.
\end{itemize} 

Similarly to Lemmas~\ref{L:nonneg_van} and \ref{L:nonpos_van}, we can precise the points where a constant sign Green's function may vanish.
\begin{lemma}\label{L:cs_delta1pos}
Let $M<\pi^2$ and $\delta_1>0$. If $G_{M,\delta_1,\delta_2}$  has constant sign on $I \times I$ and vanishes at some point $(t_0,s_0)$, then either $t_0=1$ or $t_0=s_0$.
\end{lemma}
\begin{proof}
Let us suppose that $G_{M,\delta_1,\delta_2}\ge 0$ (the case $G_{M,\delta_1,\delta_2}\le 0$ would be analogous). From Lemma~\ref{L:cs_van_gen}, we only need to discard the case $t_0=0$. Suppose then that $G_{M,\delta_1,\delta_2}(0,s_0)=0$ for some $s_0\in(0,1)$. In such a case, from the equality
\[G_{M,\delta_1,\delta_2}(0,s_0)-G_{M,\delta_1,\delta_2}(1,s_0) = \delta_1 \int_{0}^{1} G_{M,\delta_1,\delta_2}(t,s_0)\, dt > 0,\]
we deduce that $G_{M,\delta_1,\delta_2}(1,s_0)<0$, which is a contradiction. Therefore, $G_{M,\delta_1,\delta_2}$ cannot vanish at $(0,s_0)$.
\end{proof}

\subsection{Negativeness of $G_{M,\delta_1,\delta_2}$}
Now, we study the region where the Green's function is negative on the square of definition. We distinguish two situations.
\subsubsection{$M\neq 0$}
We analyze the region where $G_{M,\delta_1,\delta_2}$ is negative on $I \times I$. To do this, taking into account Corollary~\ref{cor-sign-delta1}, we fix $M\neq 0$ and $\delta_2$ for which $G_{M,0,\delta_2}$ is negative on $I \times I$, that is, $\delta_2\in \left(M,f(M)\right)$, with 
\begin{equation*}
 	f(M)=\begin{cases}
 		\frac{m^{2}}{1-\cos\left(\frac{m}{2}\right)}, & M=m^2, \ m\in (0,\pi),\\[8pt]
 		\frac{2m^{2} e^{\frac{m}{2}}}{1+e^{m}-2 e^{\frac{m}{2}}}, & M=-m^2<0, \ m\in (0,\infty).
 	\end{cases}
 \end{equation*}

Taking into account Lemma~\ref{G-symmetric}, we only need to do the calculations for $\delta_1>0$ (since case $\delta_1<0$ is followed by symmetry). 
On the other hand, it is immediate to verify that  function $\omega_1$ is strictly decreasing on $I$,  $\omega_1(0) =\frac{1}{2}$ and $\omega_1(1)=-\frac{1}{2}$.  

The characterization of the set is given on the following result.
\begin{theorem}
	Let $M<\pi^{2}, M\neq 0$ and  $\delta_2\in \left(M,f(M)\right)$, then $G_{M,\delta_1,\delta_2}$ is strictly negative on $I\times I$ if and only if   
\[|\delta_1|< 2 \,(M-\delta_2) \, G_{M,0,\delta_2}\left(1,\frac{1}{2}\right).\]
\end{theorem}
\begin{proof}
For $\delta_1>0$, it is immediately deduced from expression
\[G_{M,\delta_1,\delta_2}(t,s)=G_{M,0,\delta_2}(t,s) +\frac{\delta_1}{M-\delta_2} \, \omega_1(t)\] and the fact that $\omega_1$ attains its minimum at $t=1$ and $G_{M,0,\delta_2}$ attains its maximum at $(t,s)=\left(1,\frac{1}{2}\right)$. As a consequence	
\[\max_{t,s\in I} G_{M,\delta_1,\delta_2}(t,s)=G_{M,\delta_1,\delta_2}\left(1,\frac{1}{2}\right) =G_{M,0,\delta_2}\left(1,\frac{1}{2}\right) -\frac{\delta_1}{2\left(M-\delta_2\right)}.\] 
Thus, $G_{M,\delta_1,\delta_2}$ is negative if and only if $\delta_1<2\,(M-\delta_2) \, G_{M,0,\delta_2}\left(1,\frac{1}{2}\right)$.

The case $\delta_1<0$ follows by symmetry.
\end{proof}

\subsubsection{$M= 0$}
For the negative case we set $\delta_2\in\left(0,8\right)$ where $G_{0,0,\delta_2}$ is negative. 
\begin{theorem}\label{th:Mcero_neg}
If $M=0$ and $\delta_2\in \left(0,8\right)$, then $G_{0,\delta_1,\delta_2}$ is negative on $I\times I$ if and only if 
	\begin{center}
		$ |\delta_1|<2-\frac{\delta_2}{4}$.
	\end{center}
\end{theorem}
\begin{proof}
	Suppose that $\delta_1>0$ and let us calculate the maximum of $G_{0,\delta_1,\delta_2}$ (whose expression is given on \eqref{e-G-0}). From Lemma~\ref{L:cs_delta1pos} we know that such maximum is either at $t=1$ or $t=s$.
	
	At points of the form $\left(1,s\right)$ we have that $r(s)=G_{0,\delta_1,\delta_2}\left(1,s\right)$ has an absolute maximum at $s=\frac{1}{2}$. So, $r(s)<0$ for all $s\in I$ if and only if $r\left(\frac{1}{2}\right)<0$, that is, 	$0<\delta_1<2-\frac{\delta_2}{4}$.
	
	Let us consider now the restriction to the diagonal, that is, $h(s)=G_{0,\delta_1,\delta_2}(s,s)$.
	Given $c=\frac{1}{2}-\frac{\delta_1}{\delta_2}<\frac{1}{2}$, it holds that $h'(s)<0$ for $s<c$, $h'(c)=0$ and $h'(s)>0$ for $s>c$. Thus, $c$ is a minimum of $h$. If  $c\in (0,\frac{1}{2})$, $h$ attains its maximum either at $s=1$ or at $s=0$  while if $c\le 0$ then $h'>0$ on $(0,1]$ and the maximum is attained at $s=1$. In any case,
	$h(0)=\frac{-2-\delta_1}{2\delta_2}<0$ and $h(1)=\frac{-2+\delta_1}{2\delta_2}>h(0)$. So, $h(s)<0$ if and only if $h(1)<0$, that is, $0<\delta_1<2$.
	
Therefore, for $\delta_1>0$, $G_{M,\delta_1,\delta_2}<0$ on $I \times I$ if and only if
\[\delta_1<\min\left\{2-\frac{\delta_2}{4}, \, 2 \right\}=2-\frac{\delta_2}{4}. \]
	
Using the symmetry of $G_{0,\delta_1,\delta_2}$ with respect to $\delta_1$ we conclude the result.
\end{proof}

\subsection{Positiveness of $G_{M,\delta_1,\delta_2}$}
Let us calculate now the regions where $G_{M,\delta_1,\delta_2}$ is positive. As usual, we distinguish two cases.
\subsubsection{$M\neq 0$}
Taking into account Corollary~\ref{cor-sign-delta1}, let us fix $M\neq 0$ and $\delta_2$ such that $G_{M,0,\delta_2}$ is positive on $I \times I$, that is, $\delta_2\in \left(g(M),M\right)$, with
 \begin{equation}
 	\label{e-g(m)}
	g(M):=\begin{cases}
		g_{1}\left(\sqrt{M}\right), & M\in (0,\pi^2),\\
		g_{2}\left(\sqrt{-M}\right), & M<0,
	\end{cases}
\end{equation}
where $g_{1}(m)=-\dfrac{m^{2} \cos\left(\frac{m}{2}\right)}{1-\cos\left(\frac{m}{2}\right)}$ and $g_{2}(m)=-\frac{m^2\,\cosh\left(\frac{m}{2}\right)}{\cosh\left(\frac{m}{2}\right)-1}$.	

Now we define the function
 \begin{equation}
	\label{e-k(M)}
	k(M):=\begin{cases}
		k_{1}\left(\sqrt{M}\right), & M\in (0,\pi^2),\\
		k_{2}\left(\sqrt{-M}\right), & M<0,
	\end{cases}
\end{equation}
where $k_1(m)=-m^2\cot^2\left(\frac{m}{2}\right)$ and $k_2(m)=
-m^{2}\coth^{2}\left(\frac{m}{2}\right)$.

It is easy to verify that 
\[
g(M) <k(M)<M, \quad \mbox{for all} \quad M \neq 0.
\]

We shall consider now two different cases, depending on the sign of the parameter $M$. We start with $M>0$.

\begin{theorem}\label{th:delta1Mpos} 	Let functions $g$ and $k$ be defined in \eqref{e-g(m)} and \eqref{e-k(M)} respectively. Assume that $M=m^{2}$, $m\in (0,\pi)$ and  $\delta_2\in \left(g(M),M\right)$,  then the two following properties are fulfilled:
\begin{enumerate}
\item If $\delta_2\in (k(M),M)$ then $G_{M,\delta_1,\delta_2}$ is strictly positive on $I\times I$ if and only if
\[|\delta_1|<m\cot\left(\frac{m}{2}\right).\]
\item If $\delta_2\in (g(M),k(M)]$, then  $G_{M,\delta_1,\delta_2}$ is strictly positive on $I\times I$ if and only if 
\begin{equation*}
|\delta_1| <\frac{\sqrt{ -\delta_2^{2}+\cos^{2}\left(\frac{m}{2}\right)\left(m^{2}-\delta_2\right)^{2}}}{ m}. 
\end{equation*}
\end{enumerate}
\end{theorem}

\begin{proof}
Let us assume that $\delta_1>0$ and calculate the minimum of $G_{M,\delta_1,\delta_{2}}$. From Lemma~\ref{L:cs_delta1pos} we know that such minimum is either at $t=1$ or $t=s$.

Let us distinguish several cases:
\begin{enumerate}
\item $\delta_2\ge 0$ (that is, $\delta_2\in [0,M)$):

At the points of form $\left(1,s\right)$ the function to be studied is 
\begin{equation*}
r(s)=G_{M,\delta_1,\delta_2}\left(1,s\right)=\dfrac{\cos\left(\frac{m}{2}\left(1-2s\right)\right) \csc\left(\frac{m}{2}\right)}{2m}+\dfrac{\delta_2 \cot\left(\frac{m}{2}\right)}{2m\left(m^{2}-\delta_2\right)}-\dfrac{\delta_1}{2\left(m^{2}-\delta_2\right)},
\end{equation*}
whose minimum is attained at $s=0$ and $s=1$ (indeed, $r(0)=r(1)$). Thus, $r(s)>0$ for all $s\in I$ if and only if $r(0)=r(1)>0$, that is, $0<\delta_1<m \cot\left(\frac{m}{2}\right)$.

At the diagonal $t=s$ we obtain the following function 
\begin{equation*}
	h(s)=G_{M,\delta_1,\delta_2}\left(s,s\right)=\frac{\csc \left(\frac{m}{2}\right)}{2 m} \left(\cos \left(\frac{m}{2}\right) + \frac{\delta_1\, m \sin \left(m \left(\frac{1}{2}-s\right)\right)+\delta_2 \cos \left(m \left(\frac{1}{2}-s\right)\right)}{m^2-\delta_2}\right),
\end{equation*}
which attains its minimum at $s=1$ and so $h(s)>0$ on $I$ if and only if $h(1)=r(1)>0$. 

Thus, from Lemma \ref{L:cs_delta1pos}, we have that $G_{M,\delta_1,\delta_2}>0$ on $I \times I$ if and only if $0<\delta_1<m\cot\left(\frac{m}{2}\right)$.

\item $\delta_2<0$ (that is, $\delta_2\in (g(M),0)$):

At the points of the form $(1,s)$, analogously to the previous case, we obtain that $r(s)>0$ on $I$ if and only if $r(1)>0$, that is, $0<\delta_1<m\,\cot\left(\frac{m}{2}\right)$.

At the diagonal $t=s$, we have that $h'(c)=0$, $h'(s)>0$ for $s>c$ and $h'(s)<0$ for $s<c$, with $c=\frac{1}{2}-\frac{1}{m}\arctan\left(\frac{m\,\delta_1}{\delta_2} \right)$. So, $c$ is a minimum of $h$. Moreover, we note that $c\in I$ if and only if $\frac{1}{m} \arctan\left(\frac{m\delta_1}{\delta_2}\right)\in\left[-\frac{1}{2},0\right)$, that is, $0<\delta_1\leq -\frac{\delta_2}{m} \tan\left(\frac{m}{2}\right)$. Therefore, we subdivide the case $\delta_2<0$ into two cases:
\begin{itemize}
\item[(a)] If  $\delta_1\ge -\frac{\delta_2}{m} \tan\left(\frac{m}{2}\right)$, then $h'(s)<0$ for all $s\in [0,1)$ and the minimum of $h$ is attained at $s=1$. Thus, $h(s)>0$ on $I$ if and only if $h(1)=r(1)>0$, that is, $0<\delta_1<m\cot\left(\frac{m}{2} \right)$. As a consequence, $G_{M,\delta_1,\delta_2}>0$ on $I \times I$ for $0<\delta_1<m\cot\left(\frac{m}{2} \right)$.

We note that the two previous conditions, $\delta_1\ge -\frac{\delta_2}{m} \tan\left(\frac{m}{2}\right)$ and $0<\delta_1<m\cot\left(\frac{m}{2} \right)$, are compatible if and only if $\delta_2>-m^2\cot^2\left(\frac{m}{2}\right)\equiv k(M)$.
\item[(b)] If $0<\delta_1<-\frac{\delta_2}{m} \tan\left(\frac{m}{2}\right)$, then $h$ attains an  absolute minimum at $c\in \left(0,1\right)$. In this case, $h(c)>0$ if and only if 
\begin{equation*}
	\delta_1<\frac{\sqrt{m^{4} -2m^{2} \delta_2 -\delta_2^{2}+\left(m^{2}-\delta_2\right)^{2} \cos(m) }}{\sqrt{2}  \, m} = \frac{\sqrt{ -\delta_2^{2}+\cos^{2}\left(\frac{m}{2}\right)\left(m^{2}-\delta_2\right)^{2}}}{ m}.
\end{equation*}
We note that $-\delta_2^{2}+\cos^{2}\left(\frac{m}{2}\right)\left(m^{2}-\delta_2\right)^{2}>0$ for $m\in (0,\pi)$ and
$\delta_2\in (g(M),0)$. Indeed,  $-\delta_2^{2}+\cos^{2}\left(\frac{m}{2}\right)\left(m^{2}-\delta_2\right)^{2}>0$ if and only if  $\left|\cos\left(\frac{m}{2}\right)\right|\,|m^{2}-\delta_2|\ge |\delta_2|$. Since $m\in (0,\pi)$ and $\delta_2<0$, previous inequality is equivalent to
\[\delta_2\ge -\frac{m^{2} \cos\left(\frac{m}{2}\right)}{1-\cos\left(\frac{m}{2}\right)}\equiv g(M).\]
\end{itemize}

Moreover, we note that
\[\min\left\{-\frac{\delta_2}{m} \tan\left(\frac{m}{2}\right), \, \frac{\sqrt{ -\delta_2^{2}+\cos^{2}\left(\frac{m}{2}\right)\left(m^{2}-\delta_2\right)^{2}}}{ m} \right\}\]
\[=\begin{cases}
\frac{\sqrt{ -\delta_2^{2}+\cos^{2}\left(\frac{m}{2}\right)\left(m^{2}-\delta_2\right)^{2}}}{ m}, & \delta_2\in (g(M),k(M)), \\
-\frac{\delta_2}{m} \tan\left(\frac{m}{2}\right), & \delta_2\in(k(M),M).
\end{cases} \]

As a consequence, we conclude that:
\begin{itemize}
\item If $\delta_2\in (g(M),k(M)]$ then, from $(b)$, $G_{M,\delta_1,\delta_2}>0$ for
\[0<\delta_1<\min\left\{-\frac{\delta_2}{m} \tan\left(\frac{m}{2}\right), \, \frac{\sqrt{ -\delta_2^{2}+\cos^{2}\left(\frac{m}{2}\right)\left(m^{2}-\delta_2\right)^{2}}}{ m} \right\}\]
\[= \frac{\sqrt{ -\delta_2^{2}+\cos^{2}\left(\frac{m}{2}\right)\left(m^{2}-\delta_2\right)^{2}}}{ m} .\]
\item If $\delta_2\in (k(M),M)$ then, from $(a)$, $G_{M,\delta_1,\delta_2}>0$ for
\[-\frac{\delta_2}{m}\tan\left(\frac{m}{2}\right)<\delta_1 < m\cot\left(\frac{m}{2}\right) \]
and, from $(b)$, $G_{M,\delta_1,\delta_2}>0$ for
\[ 0<\delta_1 \le  \min\left\{-\frac{\delta_2}{m} \tan\left(\frac{m}{2}\right), \, \frac{\sqrt{ -\delta_2^{2}+\cos^{2}\left(\frac{m}{2}\right)\left(m^{2}-\delta_2\right)^{2}}}{ m} \right\} \]
\[= \frac{\sqrt{ -\delta_2^{2}+\cos^{2}\left(\frac{m}{2}\right)\left(m^{2}-\delta_2\right)^{2}}}{m}.\]
Thus, $G_{M,\delta_1,\delta_2}>0$ on $I \times I$ for $0<\delta_1<m\cot\left(\frac{m}{2}\right)$.
\end{itemize}
\end{enumerate}

The fact that the obtained bounds are optimal follows from Lemma \ref{L:cs_delta1pos}.

Using the symmetry with respect to $\delta_1$ we conclude the result.
\end{proof}

In the sequel we consider the case $M<0$.

\begin{theorem}
	Let functions $g$ and $k$ be defined in \eqref{e-g(m)} and \eqref{e-k(M)} respectively. For any $M=-m^{2}$, with $m>0$ and  $\delta_2\in (g(M),M)$, it holds that
\begin{enumerate}
\item If $\delta_2\in (g(M),k(M)]$ then  $G_{M,\delta_1,\delta_2}$ is strictly positive on $I\times I$ if and only if
\[|\delta_1|<\frac{\sqrt{\delta_2^{2}-\left(m^{2}+\delta_2\right)^{2} \cosh^{2}\left(\frac{m}{2}\right)}}{ m}. \]
\item If $\delta_2\in (k(M),M)$ then  $G_{M,\delta_1,\delta_2}$ is strictly positive on $I\times I$ if and only if
\[|\delta_1|< m \coth\left(\frac{m}{2}\right).\]
\end{enumerate}   
\end{theorem}
\begin{proof}
Let us assume that $\delta_1>0$ and calculate the minimum of $G_{M,\delta_1,\delta_{2}}$. From Lemma~\ref{L:cs_delta1pos} we know that such minimum is either at $t=1$ or $t=s$.

At the points of the form $(1,s)$ we get  the function  
\begin{equation*}
r(s)=G_{M,\delta_1,\delta_2}(1,s)=\frac{\csch\left(\frac{m}{2}\right)}{2 m} \left(-\cosh \left(m \left(s-\frac{1}{2}\right)\right) + \frac{\delta_1 m \sinh \left(\frac{m}{2}\right)+\delta_2 \cosh \left(\frac{m}{2}\right)}{m^2+\delta_2}\right),
\end{equation*}
whose minimum is attained at $s=0$ and $s=1$ (indeed, $r(0)=r(1)$). Thus, $r(s)>0$ for all $s\in I$ if and only if $r(0)=r(1)>0$, that is, $0<\delta_1<m \coth\left(\frac{m}{2}\right)$.

At the diagonal $t=s$ we obtain the following function 
\begin{equation*}
h(s)=G_{M,\delta_1,\delta_2}(s,s)=\frac{\csch\left(\frac{m}{2}\right)}{2 m} \left(-\cosh \left(\frac{m}{2}\right)+\frac{\delta_2 \cosh \left(m \left(\frac{1}{2}-s\right)\right)-\delta_1 m \sinh \left(m \left(\frac{1}{2}-s\right)\right)}{m^2+\delta_2}\right).
\end{equation*}
It occurs that $h'(s)=0$ if and only if    $\tanh\left(\frac{m}{2}\left(1-2s\right)\right)=\frac{m \delta_1}{\delta_2}$. Since  $\delta_2<0$ and $\tanh^{-1}\left(x\right)$ exists for $x\in [-1,1]$, we have that $h$ has a critical point $c=\frac{1}{2}-\frac{1}{m} \tanh^{-1}\left(\frac{m\delta_1}{\delta_2}\right)$ if and only if $-1\leq\frac{m\delta_1}{\delta_2}<0$, that is, $\delta_1\le -\frac{\delta_2}{m}$. In such a case, it occurs that $h'(s)<0$ for $s<c$ and $h'(s)>0$ for $s>c$. Moreover, we can see that $c\in I$ if and only if  $0<\delta_1\le -\frac{\delta_2}{m} \tanh\left(\frac{m}{2}\right)\left(<-\frac{\delta_2}{m}\right)$.  Therefore, we distinguish two cases:
\begin{itemize}
\item[(a)] If $\delta_1> -\frac{\delta_2}{m} \tanh\left(\frac{m}{2}\right)$, then $h'(s)<0$ for all $s\in I$ and $h$ has a minimum at $s=1$. Then $h(s)>0$ for all $s\in I$ if and only if $h(1)=r(1)>0$, that is, if and only if 	$0<\delta_1<m \coth\left(\frac{m}{2}\right)$. As a consequence,  $G_{M,\delta_1,\delta_2}>0$ on $I \times I$ for $0<\delta_1<m \coth\left(\frac{m}{2}\right)$. 

We note that the two previous conditions, $\delta_1> -\frac{\delta_2}{m} \tanh\left(\frac{m}{2}\right)$ and $0<\delta_1<m \coth\left(\frac{m}{2}\right)$, are compatible if and only if $\delta_2>-m^2\,\coth^2\left(\frac{m}{2}\right)\equiv k(M)$.
\item[(b)] If $0<\delta_1\le -\frac{\delta_2}{m} \tanh\left(\frac{m}{2}\right)$, then $h$ attains an  absolute minimum $c\in \left(0,1\right)$. In this case, $h\left(c\right)>0$ (and, consequently, $h(s)>0$ for $s\in I$) if and only if 
\begin{equation*}
	0<\delta_1<\frac{\sqrt{\delta_2^{2}-\left(m^{2}+\delta_2\right)^{2} \cosh^{2}\left(\frac{m}{2}\right)}}{ m}.
\end{equation*}
Note that, analogously to what has been done in Theorem~\ref{th:delta1Mpos}, it can be proved that  $\delta_2^{2}-\left(m^{2}+\delta_2\right)^{2} \cosh^{2}\left(\frac{m}{2}\right)>0$ for $M=-m^2<0$ and $\delta_2\in (g(M),M)$.

Therefore, since $h(c)<h(1)=r(1)$, $G_{M,\delta_1,\delta_2}>0$ on $I \times I$ for 
\begin{equation*}
	0<\delta_1<\frac{\sqrt{\delta_2^{2}-\left(m^{2}+\delta_2\right)^{2} \cosh^{2}\left(\frac{m}{2}\right)}}{ m}.
\end{equation*}
\end{itemize}

Moreover, we note that 
\[\min\left\{-\frac{\delta_2}{m} \tanh\left(\frac{m}{2}\right), \frac{\sqrt{\delta_2^{2}-\left(m^{2}+\delta_2\right)^{2} \cosh^{2}\left(\frac{m}{2}\right)}}{ m}\right\} \]
\[=\begin{cases}
\frac{\sqrt{\delta_2^{2}-\left(m^{2}+\delta_2\right)^{2} \cosh^{2}\left(\frac{m}{2}\right)}}{ m}, & \delta_2\in (g(M),k(M)),\\[4pt]
-\frac{\delta_2}{m} \tanh\left(\frac{m}{2}\right), & \delta_2\in (k(M),M). 
\end{cases}\]

As a consequence, reasoning analogously to previous theorem and using symmetry with respect to $\delta_1$, we conclude that the attained bounds are optimal and the result holds.
\end{proof}

\subsubsection{$M=0$}
As in the previous case we compute the positive sign of $G_{0,\delta_1,\delta_2}$ fixing the value of $\delta_2$. 

\begin{theorem} 
	 	Let $M=0$ and  $\delta_2\in \left(-8,0\right)$, then $G_{0,\delta_1,\delta_2}>0$ on $I\times I$ if and only if 
\begin{equation*}
 |\delta_1|<
 \begin{cases}
 	\frac{1}{2} \sqrt{-8\delta_2-\delta_2^{2}}, &  \delta_2\in(-8,-4),\\
 	2, & \delta_2\in [-4,0).
 \end{cases}
\end{equation*}
\end{theorem}
\begin{proof}
Let us assume that $\delta_1>0$ and calculate the minimum of $G_{0,\delta_1,\delta_{2}}$. From Lemma~\ref{L:cs_delta1pos} we know that such minimum is either at $t=1$ or $t=s$.

As we have seen in Theorem~\ref{th:Mcero_neg}, $r(s)=G_{0,\delta_1,\delta_2}(1,s)$ has its maximum at $s=\frac{1}{2}$ and the minimum at $s=0$ and $s=1$. Hence, $r>0$ if and only if $r(0)=r(1)>0$, that is, $0<\delta_1<2$.

On the other hand, as we have seen in Theorem~\ref{th:Mcero_neg}, $h(s)=G_{0,\delta_1,\delta_2}(s,s)$ has an absolute minimum at $c=\frac{1}{2}-\frac{\delta_1}{\delta_2}$, $h'(s)<0$ for $s<c$ and $h'(s)>0$ for $s>c$.

We distinguish the following cases:
\begin{itemize}
\item If $\delta_1\geq -\frac{\delta_2}{2}$, then $c\ge 1$ and the minimum of $h$ is attained at $s=1$ and $h(1)>0$ if and only if $\delta_1<2$. Since $h(1)=r(1)$, we deduce that if $-\frac{\delta_2}{2}\leq \delta_1<2$ then $G_{0,\delta_1,\delta_2}>0$ on $I \times I$. We note that this is only possible when $\delta_2>-4$.

\item If $0<\delta_1<-\frac{\delta_2}{2}$, then $c\in (0,1)$ and $h(c)>0$ if and only if 	$\delta_1<\frac{1}{2} \sqrt{-8\delta_2-\delta_2^{2}}$. Since $r(1)=h(1)>h(c)$, we deduce that $G_{0,\delta_1,\delta_2}>0$ for
\[0<\delta_1<\min\left\{-\frac{\delta_2}{2}, \frac{1}{2} \sqrt{-8\delta_2-\delta_2^{2}}\right\} = \begin{cases}
	\frac{1}{2} \sqrt{-8\delta_2-\delta_2^{2}}, & \delta_2\in(-8,-4),\\
	-\frac{\delta_2}{2}, & \delta_2\in[-4,0).
\end{cases}\] 
\end{itemize}

In conclusion, $G_{0,\delta_1,\delta_2}(t,s)>0$ on $I \times I$  for all $t,s\in I$ for
\[0<\delta_1<\begin{cases}
		\frac{1}{2} \sqrt{-8\delta_2-\delta_2^{2}}, & \delta_2\in(-8,-4),\\
		2, & \delta_2\in[-4,0).
	\end{cases}\]

Using the symmetry with respect to $\delta_1$ we conclude the result.
\end{proof}

\subsection{A particular case: $\delta_2=0$}
Finally, as a consequence of the previous results, we arrive at the following corollary.
\begin{corollary}
Let's consider the perturbed periodic problem
\begin{equation}\label{to}
\left\{
\begin{aligned}
u''(t)+M u(t)&=\sigma(t),\;\; t\in I,\\
u(0)-u(1)&=\delta_1 \displaystyle \int_{0}^{1} u(s) ds,\\
u'(0)-u'(1)&=0,
\end{aligned}
\right.
\end{equation}
for $M\in \mathbb{R}\setminus\{0 \}$.  The following statements holds:
\begin{enumerate}
\item  If $M=m^{2}>0$, then $G_{M,\delta_1,0}>0$ on $I \times I$  if and only if $m\in (0,\pi)$ and $|\delta_1|<m \cot\left(\frac{m}{2}\right)$.
\item If $M=-m^{2}<0$, then $G_{M,\delta_1,0}<0$ on $I \times I$  if and only if $|\delta_1|<\frac{2m e^{\frac{m}{2}}}{1-e^{m}}$.
\end{enumerate}
\end{corollary}
In this case, the graph showing the sign of the Green's function on the $(M,\delta_1)$ plane can be seen in Figure~\ref{Fig:sign_G2}.
\begin{figure}[H]
\begin{center}
	\includegraphics[width=15cm]{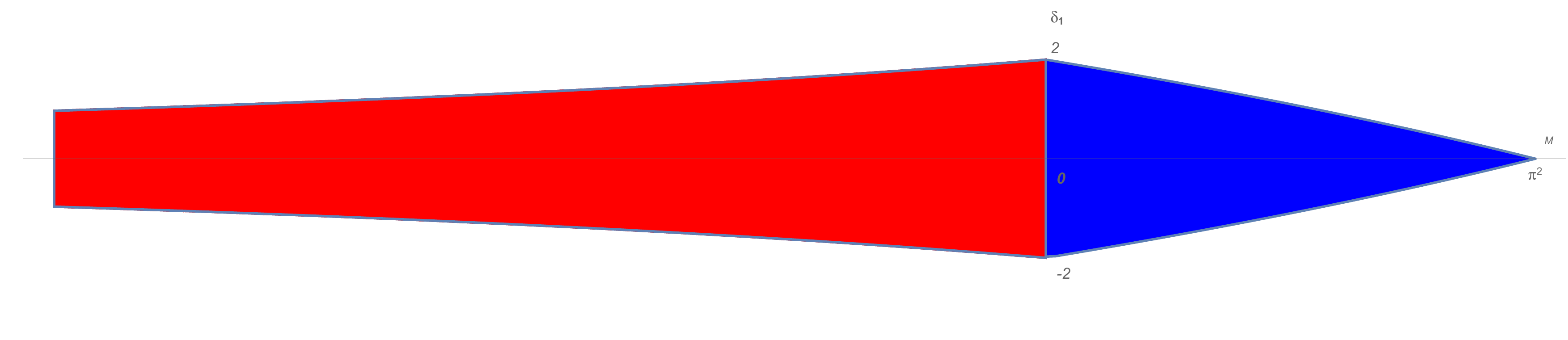}
\end{center}
\caption{The blue and red areas represent the regions of positive and negative sign of the Green's function, respectively.} \label{Fig:sign_G2}
\end{figure}

\end{document}